\documentclass[onefignum,onetabnum]{siamonline220329}

\usepackage{braket,amsfonts,amssymb}
\usepackage{physics}
\usepackage{array}
\usepackage{multirow}
\usepackage[caption=false]{subfig}
\usepackage{pgfplots}
\usepackage{algorithmic}
\usepackage{graphicx,epstopdf,psfrag,color}

\newsiamthm{claim}{Claim}
\newsiamremark{remark}{Remark}
\newsiamremark{assumption}{Assumption}
\newsiamremark{hypothesis}{Hypothesis}
\crefname{hypothesis}{Hypothesis}{Hypotheses}

\Crefname{ALC@unique}{Line}{Lines}

\usepackage{amsopn}

\newcommand{\delths}{\Delta\theta^*}
\newcommand{\delthsbytwo}{\frac{\Delta\theta^*}{2}}
\newcommand{\delk}{\delta_k}

\title{A Two-Dimensional Map of Arbitrary Period\thanks{Submitted to the editors DATE.
\funding{NSF Grant CMMI-2043464}}}

\author{Aakash Khandelwal\thanks{Department of Mechanical Engineering, Michigan State University, East Lansing, MI (\email{khande10@egr.msu.edu}, \email{mukherji@egr.msu.edu}).}
\and Ranjan Mukherjee\footnotemark[2]}

\headers{A Two-Dimensional Map of Arbitrary Period}{Aakash Khandelwal, Ranjan Mukherjee}

\ifpdf
\hypersetup{ pdftitle={A Two-Dimensional Map of Arbitrary Period} }
\fi

\begin{document}
\maketitle

\begin{abstract}

We introduce a two-dimensional discrete-time dynamical system which represents the evolution of an angle and angular velocity. While the angle evolves by a fixed amount in every step, the evolution of the angular velocity is governed by a nonlinear map. We study the periodicity and stability of solutions to the system for a range of parameter values and initial conditions. The coupled system is shown to be periodic for certain parameter choices and initial conditions. In the limit, when the change in the angle tends to zero, the map is equivalent to the dynamics of a simple pendulum. Based on the integral of motion of the pendulum, an approximate invariant for the system is obtained. Simulations showing the behavior of the system for different parameter values and initial conditions are presented.
\end{abstract}



\section{Introduction}
\label{sec:intro}

We consider the following two-dimensional discrete-time autonomous dynamical system in the states $\theta \in S^1$, $S^1 = R \mod 2\pi$, and $\omega \in R^+$, which represent an angle and angular velocity respectively:
\begin{align} 
    \theta_{k+1} &= \theta_k + \delths \label{eq:thkp1-zero-dyn} \\
    \omega_{k+1} - \omega_k &= P \sin\theta_k \left[ \frac{1}{\omega_k} + \frac{1}{\omega_{k+1}} \right] \label{eq:omkp1-zero-dyn-reduced}
\end{align}

\noindent where $\delths \in (0, \pi)$, and the parameter $P$ is given by
\begin{equation} \label{eq:P}
    P \triangleq \pm \frac{g \delths{^2}}{2\ell \sin(\delths)}
\end{equation}

\noindent with $g$ denoting the acceleration due to gravity and $\ell$ denoting some length. It can be verified that $P$ has the unit of s$^{-2}$ and this makes \eqref{eq:omkp1-zero-dyn-reduced} dimensionally consistent. The map \eqref{eq:thkp1-zero-dyn} is the linear, unperturbed circle map \cite{gilmore_topology_2011} and its behavior is well understood, but the map \eqref{eq:omkp1-zero-dyn-reduced} is nonlinear. To our knowledge, the coupled system in \eqref{eq:thkp1-zero-dyn} and \eqref{eq:omkp1-zero-dyn-reduced} has not appeared in the literature. We are interested in the periodicity and stability of solutions to the system, and explore the existence of invariants along solutions to the system.

The existence of periodic solutions for several classes of difference equations has been reviewed in \cite{dannan_periodic_2000}. Conditions for the existence of periodic solutions to some classes of second-order dynamical systems are considered in \cite{li_periodic_2009, maroncelli_periodic_2016}, and the periodic character and solutions of rational difference equations has been studied in \cite{patula_oscillation_2002, cima_periodic_2004, bajo_periodicity_2006, stevic_short_2004, stevic_global_2006, stevic_system_2011}. Invariants for difference equations are studied in \cite{schinas_invariants_1997, papaschinopoulos_invariants_2001}. None of these results directly apply to the system under consideration owing to the structure of \eqref{eq:omkp1-zero-dyn-reduced}.\

Two dimensional maps such as the H\'enon map \cite{henon_two-dimensional_1976, hitzl_exploration_1985} are known to exhibit fixed points, stable periodic orbits, and strange attractors based on parameter choices. However, the resulting solution is not immediately obvious from the parameter choices. Since \eqref{eq:thkp1-zero-dyn} may be periodic with a known period depending on the choice of $\delths$, we study the existence of periodic solutions to the coupled system. It is proved that, when $\delths$ is an integer submultiple of $2\pi$, the system permits periodic orbits for specific initial values of $\theta$ and a large range of initial values of $\omega$. There are infinitely many such orbits, each of which is stable but not attractive. For different initial values of $\theta$, simulations show that periodicity is lost.

In the limit, when $\delths$ is infinitesimally small, our discrete-time dynamical system is equivalent to a continuous time system with the dynamics of a simple pendulum. Motivated by the existence of an integral of motion for the simple pendulum, we propose a quantity that is approximately invariant along solutions of the dynamical system in \eqref{eq:thkp1-zero-dyn} and \eqref{eq:omkp1-zero-dyn-reduced}, and is equivalent to the integral of motion for the pendulum in the limiting case. The approximately invariant quantity allows $\omega$ to be expressed as an explicit periodic function of $\theta$, which is not possible from \eqref{eq:omkp1-zero-dyn-reduced}. For sufficiently large $\omega$, this periodic function closely approximates the behavior of the original dynamical system. Extensive simulation results show the periodic and aperiodic motion of the system for multiple choices of $\delths$ and initial conditions.

\section{Analysis of the Dynamical System} \label{sec:dynamical-sys}

The map \eqref{eq:thkp1-zero-dyn} represents a rational rotation if $\delths = (p/q)2\pi$, where $p, q \in Z^+$. If $\gcd(p, q) = 1$, \eqref{eq:thkp1-zero-dyn} is periodic with period $q$ since
\begin{equation} \label{eq:th-rational-periodicity}
    \theta_{k+q} = \theta_k + 2\pi p = \theta_k
\end{equation}

\noindent and the state $\theta$ returns to its original value after $p$ revolutions around the circle. In particular, if $p = 1$ and $q = N$, $\delths = 2\pi/N$ and \eqref{eq:thkp1-zero-dyn} is periodic with period $N$, where $\theta$ returns to its original value after a single revolution. If $\delths$ is an irrational submultiple of $2\pi$, \eqref{eq:thkp1-zero-dyn} represents an irrational rotation and exhibits quasiperiodic motion \cite{gilmore_topology_2011, turer_dynamical_nodate}. It should be noted that we study the dynamical system subject to the condition $\delths \in (0, \pi)$.

The map \eqref{eq:omkp1-zero-dyn-reduced} can be expressed as a quadratic equation in the variable $\omega_{k+1}$:
\begin{equation}
    \omega_k \omega_{k+1}^2 - (\omega_k^2 + P \sin \theta_k) \omega_{k+1} - P \sin \theta_k \omega_k = 0
\end{equation}

\noindent which can be solved for $\omega_{k+1}$ as
\begin{equation} \label{eq:omkp1-quadratic-sol}
    \omega_{k+1} = \frac{\omega_k}{2} + \frac{P \sin\theta_k}{2 \omega_k} \pm \frac{\omega_k}{2} \sqrt{1 + \frac{6 P \sin\theta_k}{\omega_k^2} + \frac{P^2 \sin^2\theta_k}{\omega_k^4}}
\end{equation}

\noindent The solution is real when
\begin{equation} \label{eq:bound-omk-real-sol}
    1 + \frac{6 P \sin\theta_k}{\omega_k^2} + \frac{P^2 \sin^2\theta_k}{\omega_k^4} \geq 0 \qquad \forall k
\end{equation}

\begin{remark} \label{rem2.1}
When $P$ is positive (negative), the above condition is satisfied $\forall \theta_k \in [0, \pi]$ ($\forall \theta_k \in [\pi, 2\pi]$) regardless of the value of $\omega_k$, since $P\sin\theta_k$ is always positive.
\end{remark}

It must now be ascertained whether the positive or negative square root must be used in the quadratic solution \eqref{eq:omkp1-quadratic-sol}. In particular, if $\theta_k \in \{0, \pi\}$, it follows from \eqref{eq:omkp1-zero-dyn-reduced} that $\omega_{k+1} = \omega_k$, and this solution is obtained by selecting the positive square root in \eqref{eq:omkp1-quadratic-sol}. Selecting the negative square root gives $\omega_{k+1} = 0$, which is incorrect.

For any value of $\theta_k$, if $\omega_k$ is sufficiently large such that
\begin{equation} \label{eq:bound-omk-real-sol-taylor}
    \abs{\frac{6 P \sin\theta_k}{\omega_k^2} + \frac{P^2 \sin^2\theta_k}{\omega_k^4}} < 1
\end{equation}

\noindent we can expand the square root in \eqref{eq:omkp1-quadratic-sol} as a convergent Taylor series expansion which gives the following expressions for the positive and negative square root solutions:
\begin{equation} \label{eq:omkp1-final-taylor}
    \omega_{k+1} = \omega_k + \frac{2 P \sin\theta_k}{\omega_k} - \frac{2 P^2 \sin^2\theta_k}{\omega_k^3} + \frac{6 P^3 \sin^3\theta_k}{\omega_k^5} - \frac{22 P^4 \sin^4\theta_k}{\omega_k^7} + \order{\frac{P^5 \sin^5\theta_k}{\omega_k^{9}}}
\end{equation}

\begin{equation}
    \omega_{k+1} = - \frac{P \sin\theta_k}{\omega_k} + \frac{2 P^2 \sin^2\theta_k}{\omega_k^3} - \frac{6 P^3 \sin^3\theta_k}{\omega_k^5} + \frac{22 P^4 \sin^4\theta_k}{\omega_k^7} + \order{\frac{P^5 \sin^5\theta_k}{\omega_k^{9}}}
\end{equation}

\noindent It can be seen from the above two equations that, for sufficiently large values of $\omega_k$, selecting the negative square root gives values of $\omega_{k+1}$ which have magnitudes much smaller than $\omega_k$ and may violate the condition $\omega \in R^+$. In this paper, we focus on solutions obtained using the positive square root, which always ensures a feasible solution. Thus, the dynamics in \eqref{eq:omkp1-quadratic-sol} is rewritten as
\begin{equation} \label{eq:omkp1-final}
    \omega_{k+1} = \frac{\omega_k}{2} + \frac{P \sin\theta_k}{2 \omega_k} + \frac{\omega_k}{2} \sqrt{1 + \frac{6 P \sin\theta_k}{\omega_k^2} + \frac{P^2 \sin^2\theta_k}{\omega_k^4}}
\end{equation}

\noindent The above equation expresses $\omega_{k+1}$ explicitly in terms of $\theta_k$ and $\omega_{k}$ and will be used in our simulations.

\section{Periodicity for Specific Initial Conditions} \label{sec:proof-periodicity}

We begin with the following assumption which guarantees the existence of periodic solutions to \eqref{eq:thkp1-zero-dyn} and \eqref{eq:omkp1-zero-dyn-reduced}.

\begin{assumption} \label{asm:omk-min-periodic}
For any initial condition $(\theta_1, \omega_1)$,
\begin{equation*} 
    \omega_{k, \mathrm{min}}^2 > |P|, \qquad  \omega_{k, \mathrm{min}} \triangleq \min_{k} \omega_k
    \end{equation*}
\end{assumption}

\begin{lemma} \label{lem:zero}
The dynamical system in \eqref{eq:thkp1-zero-dyn} and \eqref{eq:omkp1-zero-dyn-reduced} is periodic with period $N$ for $\theta_1 = 0$ and initial conditions satisfying Assumption \ref{asm:omk-min-periodic} when 
\begin{equation} \label{eq:delths}
   \delths = \frac{2\pi}{N}, \quad N \in Z^+, \quad N \geq 3
\end{equation}
\end{lemma}

\begin{proof}
It follows from \eqref{eq:thkp1-zero-dyn} and \eqref{eq:delths} that $\theta_{N+1} = \theta_1 + 2\pi = \theta_1$, which establishes periodicity of $\theta$. For any value of $N$ in \eqref{eq:delths}, we can express $\theta_k$ as follows
\begin{equation} \label{eq:thk-from-pi}
    \theta_k = (k-1) \delths = \pi + (2k - N - 2) \frac{\delths}{2}
\end{equation}

\noindent from which it follows that
\begin{equation} \label{eq:sin-from-pi}
    \sin\theta_k = - \sin\left[(2k - N - 2) \frac{\delths}{2}\right]
\end{equation}

\noindent Using \eqref{eq:sin-from-pi}, the dynamics \eqref{eq:omkp1-zero-dyn-reduced} may be rewritten as
\begin{equation} \label{eq:omkp1-zero-dyn-from-pi}
    \omega_{k+1} - \omega_k = - P \sin\left[(2k - N - 2) \frac{\delths}{2}\right] \left[ \frac{1}{\omega_k} + \frac{1}{\omega_{k+1}} \right]
\end{equation}

\noindent Substituting $k = 1$, we get
\begin{equation} \label{eq:om2-eq-om1}
    \omega_{2} = \omega_1
\end{equation}

\noindent We prove periodicity of $\omega$ separately for even and odd values of $N$.

\paragraph{Even $N$}

Substituting $k = \frac{N}{2}+1$ in \eqref{eq:thk-from-pi} and \eqref{eq:omkp1-zero-dyn-from-pi}, we get $\theta_{\frac{N}{2}+1} = \pi$ and $\omega_{\frac{N}{2}+2} - \omega_{\frac{N}{2}+1} = 0$, from which it follows that
\begin{equation} \label{eq:equal-pi}
    \omega_{\frac{N}{2}+2} = \omega_{\frac{N}{2}+1}
\end{equation}

\noindent Choosing $k = \left(\frac{N}{2}+1\right) - 1$ and $k = \left(\frac{N}{2}+1\right) + 1$ in \eqref{eq:omkp1-zero-dyn-from-pi}, we obtain
\begin{align} 
    \omega_{\frac{N}{2}+1} - \omega_{\frac{N}{2}} &= P \sin(\delths) \left[ \frac{1}{\omega_{\frac{N}{2}}} + \frac{1}{\omega_{\frac{N}{2}+1}} \right] \label{eq:even-Nby2} \\
    \omega_{\frac{N}{2}+3} - \omega_{\frac{N}{2}+2} &= - P \sin(\delths) \left[ \frac{1}{\omega_{\frac{N}{2}+2}} + \frac{1}{\omega_{\frac{N}{2}+3}} \right] \label{eq:even-Nby2p2}
\end{align}

\noindent Adding the above two equations, we obtain
\begin{equation}
    \omega_{\frac{N}{2}+3} - \omega_{\frac{N}{2}+2} + \omega_{\frac{N}{2}+1} - \omega_{\frac{N}{2}} = P \sin(\delths) \left[ \frac{1}{\omega_{\frac{N}{2}}} + \frac{1}{\omega_{\frac{N}{2}+1}} - \frac{1}{\omega_{\frac{N}{2}+2}} - \frac{1}{\omega_{\frac{N}{2}+3}} \right]
\end{equation}

\noindent Using \eqref{eq:equal-pi}, the above equation simplifies to
\begin{equation}
    \omega_{\frac{N}{2}+3} - \omega_{\frac{N}{2}} = P \sin(\delths) \left[ \frac{\omega_{\frac{N}{2}+3} - \omega_{\frac{N}{2}}}{\omega_{\frac{N}{2}} \omega_{\frac{N}{2}+3}} \right]
\end{equation}
From Assumption \ref{asm:omk-min-periodic}, it follows
\begin{equation} \label{eq:equal-delths}
    \omega_{\frac{N}{2}+3} = \omega_{\frac{N}{2}}
\end{equation}

\noindent Similarly, for $k = \left(\frac{N}{2}+1\right) - 2$ and $k = \left(\frac{N}{2}+1\right) + 2$ in \eqref{eq:omkp1-zero-dyn-from-pi}, we obtain
\begin{align} 
    \omega_{\frac{N}{2}} - \omega_{\frac{N}{2}-1} &= P \sin(2\delths) \left[ \frac{1}{\omega_{\frac{N}{2}-1}} + \frac{1}{\omega_{\frac{N}{2}}} \right] \label{eq:even-Nby2m1} \\
    \omega_{\frac{N}{2}+4} - \omega_{\frac{N}{2}+3} &= - P \sin(2\delths) \left[ \frac{1}{\omega_{\frac{N}{2}+3}} + \frac{1}{\omega_{\frac{N}{2}+4}} \right] \label{eq:even-Nby2p3}
\end{align}

\noindent Adding the above two equations, we obtain
\begin{equation}
    \omega_{\frac{N}{2}+4} - \omega_{\frac{N}{2}+3} + \omega_{\frac{N}{2}} - \omega_{\frac{N}{2}-1} = P \sin(2\delths) \left[ \frac{1}{\omega_{\frac{N}{2}-1}} + \frac{1}{\omega_{\frac{N}{2}}} - \frac{1}{\omega_{\frac{N}{2}+3}} - \frac{1}{\omega_{\frac{N}{2}+4}} \right]
\end{equation}

\noindent Using \eqref{eq:equal-delths}, the above equation simplifies to
\begin{equation}
    \omega_{\frac{N}{2}+4} - \omega_{\frac{N}{2}-1} = P \sin(2\delths) \left[ \frac{\omega_{\frac{N}{2}+4} - \omega_{\frac{N}{2}-1}}{\omega_{\frac{N}{2}-1} \omega_{\frac{N}{2}+4}} \right]
\end{equation}
From Assumption \ref{asm:omk-min-periodic}, it follows
\begin{equation} \label{eq:equal-2delths}
    \omega_{\frac{N}{2}+4} = \omega_{\frac{N}{2}-1}
\end{equation}

\noindent This process can be repeated for all pairs of values of $k = \left(\frac{N}{2}+1\right) - n$ and $k = \left(\frac{N}{2}+1\right) + n$, $n = 0, 1, 2, \dots, \left(\frac{N}{2}-1\right)$ to establish that 
\begin{equation} \label{eq:om-pairs-even}
    \omega_{\frac{N}{2}+n+2} = \omega_{\frac{N}{2}-n+1}, \quad n = 0, 1, 2, \dots, \left(\frac{N}{2}-1\right)
\end{equation}

\noindent In particular, for $n = \left(\frac{N}{2}-1\right)$, we get
\begin{equation}
    \omega_{N+1} = \omega_2
\end{equation}

\noindent Since $\omega_2 = \omega_1$ from \eqref{eq:om2-eq-om1}, we have
\begin{equation} \label{eq:om-periodicity-even}
    \omega_{N+1} = \omega_1
\end{equation}

\paragraph{Odd $N$}

Substituting $k = \frac{N+1}{2}$ and $k = \frac{N+3}{2}$ in \eqref{eq:omkp1-zero-dyn-from-pi}, we obtain
\begin{align} 
    \omega_{\frac{N+3}{2}} - \omega_{\frac{N+1}{2}} &= P \sin\left(\frac{\delths}{2}\right) \left[ \frac{1}{\omega_{\frac{N+1}{2}}} + \frac{1}{\omega_{\frac{N+3}{2}}} \right] \label{eq:odd-Np1by2} \\
    \omega_{\frac{N+5}{2}} - \omega_{\frac{N+3}{2}} &= - P \sin\left(\frac{\delths}{2}\right) \left[ \frac{1}{\omega_{\frac{N+3}{2}}} + \frac{1}{\omega_{\frac{N+5}{2}}} \right] \label{eq:odd-Np3by2}
\end{align}

\noindent Adding the above two equations, we obtain
\begin{equation}
    \omega_{\frac{N+5}{2}} - \omega_{\frac{N+1}{2}} = P \sin\left(\frac{\delths}{2}\right) \left[ \frac{1}{\omega_{\frac{N+1}{2}}} - \frac{1}{\omega_{\frac{N+5}{2}}} \right]
\end{equation}
which simplifies to
\begin{equation}
    \omega_{\frac{N+5}{2}} - \omega_{\frac{N+1}{2}} = P \sin\left(\frac{\delths}{2}\right) \left[ \frac{\omega_{\frac{N+5}{2}} - \omega_{\frac{N+1}{2}}}{\omega_{\frac{N+1}{2}}\omega_{\frac{N+5}{2}}} \right]
\end{equation}
From Assumption \ref{asm:omk-min-periodic}, it follows
\begin{equation} \label{eq:equal-delthsby2}
    \omega_{\frac{N+5}{2}} = \omega_{\frac{N+1}{2}}
\end{equation}

\noindent Similarly, choosing $k = \frac{N+1}{2}-1$ and $k = \frac{N+3}{2}+1$ in \eqref{eq:omkp1-zero-dyn-from-pi}, we obtain
\begin{align} 
    \omega_{\frac{N+1}{2}} - \omega_{\frac{N-1}{2}} &= P \sin\left(\frac{3\delths}{2}\right) \left[ \frac{1}{\omega_{\frac{N-1}{2}}} + \frac{1}{\omega_{\frac{N+1}{2}}} \right] \label{eq:odd-Nm1by2} \\
    \omega_{\frac{N+7}{2}} - \omega_{\frac{N+5}{2}} &= - P \sin\left(\frac{3\delths}{2}\right) \left[ \frac{1}{\omega_{\frac{N+5}{2}}} + \frac{1}{\omega_{\frac{N+7}{2}}} \right] \label{eq:odd-Np5by2}
\end{align}

\noindent Adding the above two equations, we obtain
\begin{equation}
    \omega_{\frac{N+7}{2}} - \omega_{\frac{N+5}{2}} + \omega_{\frac{N+1}{2}} - \omega_{\frac{N-1}{2}} = P \sin\left(\frac{3\delths}{2}\right) \left[ \frac{1}{\omega_{\frac{N-1}{2}}} + \frac{1}{\omega_{\frac{N+1}{2}}} - \frac{1}{\omega_{\frac{N+5}{2}}} - \frac{1}{\omega_{\frac{N+7}{2}}} \right]
\end{equation}

\noindent Substituting \eqref{eq:equal-delthsby2} in the above equation and simplifying, we obtain
\begin{equation}
    \omega_{\frac{N+7}{2}} - \omega_{\frac{N-1}{2}} = P \sin\left(\frac{3\delths}{2}\right) \left[ \frac{\omega_{\frac{N+7}{2}} - \omega_{\frac{N-1}{2}}}{\omega_{\frac{N-1}{2}}\omega_{\frac{N+7}{2}}} \right]
\end{equation}
Using Assumption \ref{asm:omk-min-periodic}, we get
\begin{equation} \label{eq:equal-3delthsby2}
    \omega_{\frac{N+7}{2}} = \omega_{\frac{N-1}{2}}
\end{equation}

\noindent This process can be repeated for all pairs of values of $k = \frac{N+1}{2} - n$, $k = \frac{N+3}{2} + n$, $n = 0, 1, 2, \dots, \left(\frac{N-3}{2}\right)$ to establish that 
\begin{equation}  \label{eq:om-pairs-odd}
    \omega_{\frac{N+3}{2}+n+1} = \omega_{\frac{N+1}{2}-n}, \quad n = 0, 1, 2, \dots, \left(\frac{N-3}{2}\right)
\end{equation}

\noindent In particular, if $n = \left(\frac{N-3}{2}\right)$, the above equation gives
\begin{equation}
    \omega_{N+1} = \omega_2
\end{equation}

\noindent Since $\omega_2 = \omega_1$ from \eqref{eq:om2-eq-om1}, we have
\begin{equation} \label{eq:om-periodicity-odd}
    \omega_{N+1} = \omega_1
\end{equation}

\noindent This concludes the proof.
\end{proof}

\begin{lemma} \label{lem:delths}
The dynamical system in \eqref{eq:thkp1-zero-dyn} and \eqref{eq:omkp1-zero-dyn-reduced} is periodic with period $N$ for $\theta_1 = \delths/2$ and initial conditions satisfying Assumption \ref{asm:omk-min-periodic} when $\delths$ is given by \eqref{eq:delths}.
\end{lemma}
\begin{proof}
The proof is similar to the one for Lemma \ref{lem:zero}, and is provided in Appendix \ref{app:proof}
\end{proof}

\begin{theorem} \label{thm:periodic}
The dynamical system in \eqref{eq:thkp1-zero-dyn} and \eqref{eq:omkp1-zero-dyn-reduced} is periodic with period $N$ for $\theta_1 \in \{0, \delths/2, \delths, \dots, 2\pi-\delths/2\}$ and initial conditions satisfying Assumption \ref{asm:omk-min-periodic} when $\delths$ is given by \eqref{eq:delths}.
\end{theorem}
\begin{proof}
The initial conditions $\theta_1 \in \{0, \delths/2\}$ are covered by Lemmas \ref{lem:zero} and \ref{lem:delths}. For the other initial conditions $\theta_1 \in \{\delths, 3\delths/2, \dots, 2\pi-\delths/2\}$, equivalent initial conditions with $\theta_1 \in \{0, \delths/2\}$ can be found by decrementing $k$ while solving \eqref{eq:thkp1-zero-dyn} and \eqref{eq:omkp1-zero-dyn-reduced}.
\end{proof}

\begin{corollary} \label{cor:stable}
    The periodic orbits described in Theorem \ref{thm:periodic} are stable but not attractive.
\end{corollary}

\begin{proof}
    The proof follows directly from Theorem \ref{thm:periodic}. Let the periodic solution be given by the points $(\theta_1, \omega_1)$, $(\theta_2, \omega_2)$, \dots, $(\theta_N, \omega_N)$. If any periodic point $(\theta_k, \omega_k)$ is perturbed to $(\theta_k, \omega_k + \epsilon_k)$, the equivalent initial condition $(\theta_1, \omega_1 + \epsilon_1)$ can be found by decrementing $k$ while solving \eqref{eq:thkp1-zero-dyn} and \eqref{eq:omkp1-zero-dyn-reduced}. Since the new initial condition satisfies Assumption \ref{asm:omk-min-periodic}, it results in a distinct periodic orbit with $(\theta_{N+1}, \omega_{N+1}) = (\theta_1, \omega_1 + \epsilon_1)$.
\end{proof}
 
It follows from Corollary \ref{cor:stable} that the eigenvalues of
\begin{equation} \label{eq:floquet-matrix}
    M = J_N \cdots J_2\, J_1, \qquad
    J_k(\theta_k, \omega_k) \triangleq \begin{bmatrix}
        \dfrac{\partial\theta_{k+1}}{\partial\theta_k} & \dfrac{\partial\theta_{k+1}}{\partial\omega_k} \\
        \dfrac{\partial\omega_{k+1}}{\partial\theta_k} & \dfrac{\partial\omega_{k+1}}{\partial\omega_k}
    \end{bmatrix}
\end{equation}

\noindent have magnitude equal to 1. This will be verified numerically in simulations.

\section{Approximate Invariant for Dynamical System}

In the previous section, periodicity of the system was established for the parameter choice \eqref{eq:delths} under specific initial conditions. We now seek to describe the behavior of the system for arbitrary values of $\delths$ and arbitrary initial conditions by approximating $\omega_k$ as an explicit periodic function of $\theta_k$. 

\subsection{Limiting case of the dynamical system}

We consider the dynamical system \eqref{eq:thkp1-zero-dyn} and \eqref{eq:omkp1-zero-dyn-reduced} in the limit $\delths \rightarrow 0$. First, \eqref{eq:thkp1-zero-dyn} may be rewritten as
\begin{equation}
    \frac{\theta_{k+1} - \theta_k}{\delk} = \frac{\delths}{\delk}
\end{equation}

\noindent where $\delk$ is the time taken for the states to evolve from $(\theta_k, \omega_k)$ to $(\theta_{k+1}, \omega_{k+1})$. As $\delths \rightarrow 0$, $\delk \rightarrow 0$, and the above equation becomes
\begin{equation} \label{eq:th-lim}
    \dot \theta = \omega
\end{equation}

\noindent where $\theta$ and $\omega$ are the continuous-time analogs of the discrete variables $\theta_k$ and $\omega_k$. Dividing both sides of \eqref{eq:omkp1-zero-dyn-reduced} by $\delk$ and regrouping terms, we obtain
\begin{equation} \label{diffeqnomega}
    \frac{\omega_{k+1} - \omega_k}{\delk} = \pm \frac{g}{\ell}\, \frac{\delths}{\sin(\delths)} \frac{\delths}{\delk} \frac{1}{2} \left[ \frac{1}{\omega_k} + \frac{1}{\omega_{k+1}} \right] \sin \theta_k
\end{equation}

\noindent As $\delths \rightarrow 0$ and $\delk \rightarrow 0$, we have
\begin{equation*}
\lim_{\delths \rightarrow 0}\, \frac{\delths}{\sin(\delths)} = 1, \quad \left[ \frac{1}{\omega_k} + \frac{1}{\omega_{k+1}}\right] = \dfrac{2}{\omega}
\end{equation*}

\noindent Using the above expressions in \eqref{diffeqnomega}, we get
\begin{equation} \label{eq:simple-pendulum}
    \dot \omega = \pm \frac{g}{\ell} \sin \theta, \quad \Rightarrow \quad \ddot \theta = \pm \frac{g}{\ell} \sin \theta
\end{equation}

\noindent If the negative choice of $P$ in \eqref{eq:P} is considered, the above equation that of a simple pendulum of length $\ell$. Equation \eqref{eq:simple-pendulum} has an integral of motion given by \cite{perram_explicit_2003-1, shiriaev_constructive_2005}
\begin{equation} \label{eq:integral-of-motion-pendulum}
    E = \frac{1}{2}\dot\theta^2 \pm \frac{g}{\ell} \cos\theta 
\end{equation}

\noindent where the +ve (-ve) sign in the above equation corresponds to the +ve (-ve) sign in \eqref{eq:simple-pendulum}. With the value of $E$ computed from initial conditions, the above equation can be used to express $\omega$ as a periodic function of $\theta$
\begin{equation} \label{eq:om-periodic-pendulum}
    \omega = \sqrt{2\left[ E \mp \frac{g}{\ell} \cos\theta \right]}
\end{equation}

\noindent which provides solutions for $\omega \in R^+$ as long as the initial conditions ensure $E > g/\ell$.

\subsection{Approximation of $\omega_k$ as an explicit periodic function of $\theta_k$}

Motivated by the discussion in the preceding section, we seek a quantity similar to \eqref{eq:integral-of-motion-pendulum} for the discrete dynamical system described by \eqref{eq:thkp1-zero-dyn} and \eqref{eq:omkp1-final}, that is \emph{approximately} invariant along system trajectories. We assume the following form for this quantity:
\begin{equation} \label{eq:E-discrete}
    \bar E = \frac{1}{2}\omega_k^2 + \sigma \frac{g}{\ell} \cos\left( \theta_k - \delthsbytwo \right)
\end{equation}

\noindent which can be rewritten to express $\omega_k$ as
\begin{equation} \label{eq:om-periodic-discrete}
    \omega_k = \sqrt{2\left[ \bar E - \sigma \frac{g}{\ell} \cos\left( \theta_k - \delthsbytwo \right) \right]}
\end{equation}

\noindent In the above equations, $\sigma$ is a parameter to be determined; it is +ve (-ve) when $P$ in \eqref{eq:P} is +ve (-ve). It must be ensured that $\lim_{\delths \rightarrow 0} |\sigma| = 1$ so that, in the limit $\delths \rightarrow 0$, \eqref{eq:E-discrete} and \eqref{eq:om-periodic-discrete} reduce to the exact relations \eqref{eq:integral-of-motion-pendulum} and \eqref{eq:om-periodic-pendulum}. The argument $\theta_k - \delths/2$ of the cosine function in \eqref{eq:om-periodic-discrete} ensures $\omega_{k+1} = \omega_k$ when $\theta_k \in \{0, \pi\}$ - refer to the discussion after Remark \ref{rem2.1}. Further, if \eqref{eq:thk-from-pi} holds (this corresponds to the initial condition $\theta_1 = 0$), \eqref{eq:om-periodic-discrete} can be rewritten as:
\begin{equation*} 
    \omega_k = \sqrt{2\left[ \bar E + \sigma \frac{g}{\ell} \cos\left\{ (2k - N - 3)\delthsbytwo \right\} \right]}
\end{equation*}

\noindent which implies \eqref{eq:om-pairs-even} or \eqref{eq:om-pairs-odd} depending on whether $N$ is even or odd. If \eqref{eq:thk-from-pi-delths} holds (this corresponds to the initial condition $\theta_1 = \delths/2$), \eqref{eq:om-periodic-discrete} can be rewritten as:
\begin{equation*} 
    \omega_k = \sqrt{2\left[ \bar E + \sigma \frac{g}{\ell} \cos\left\{ (2k - N - 2)\delthsbytwo \right\} \right]}
\end{equation*}

\noindent which implies \eqref{eq:om-pairs-even-delths} or \eqref{eq:om-pairs-odd-delths}, depending on whether $N$ is even or odd.

Given a value of $\bar E$ in \eqref{eq:om-periodic-discrete}, the minima (maxima) of $\omega_k$ occur at $\theta_k = \delths/2$ and the maxima (minima) at $\theta_k = \pi + \delths/2$ when $\sigma$ is +ve (-ve).

\subsection{Comparison with exact solution}

The value of $\omega_{k+1}$ predicted by \eqref{eq:om-periodic-discrete} is given by
\begin{equation}
    \omega_{k+1} = \sqrt{2\left[ \bar E - \sigma \frac{g}{\ell} \cos\left( \theta_{k+1} - \delthsbytwo \right) \right]} = \sqrt{2\left[ \bar E - \sigma \frac{g}{\ell} \cos\left( \theta_k + \delthsbytwo \right) \right]}
\end{equation}

\noindent where \eqref{eq:thkp1-zero-dyn} was used to replace $\theta_{k+1}$ in terms of $\theta_k$. Substituting the expression for $\bar E$ from \eqref{eq:E-discrete}, we obtain
\begin{equation}
\begin{split}
    \omega_{k+1} &= \sqrt{\omega_k^2 - \sigma \frac{2 g}{\ell} \left[ \cos\left( \theta_k + \delthsbytwo \right) - \cos\left( \theta_k - \delthsbytwo \right) \right]} \\
    &= \sqrt{\omega_k^2 + \sigma \frac{4 g}{\ell} \sin\left( \delthsbytwo \right) \sin\theta_k}
    \end{split}
\end{equation}

\noindent If the value of $\sigma$ is chosen\footnote{This value is found by trial and error to minimize the error between the series expansions \eqref{eq:omkp1-from-E-taylor} and \eqref{eq:omkp1-final-taylor}. In particular, this choice makes the terms of $\order{\frac{P \sin\theta_k}{\omega_k}}$ and $\order{\frac{P^2 \sin^2\theta_k}{\omega_k^3}}$ identical in the two series.} to be
\begin{equation} \label{eq:L}
    \sigma = \pm \frac{\delths{^2}}{2 \sin(\delths/2) \sin(\delths)}
\end{equation}
\noindent which indeed satisfies $\lim_{\delths \rightarrow 0} |\sigma| = 1$, the above equation can be rewritten as
\begin{equation} \label{eq:omkp1-from-E}
    \omega_{k+1} = \omega_k \sqrt{1 + \frac{4 P \sin\theta_k}{\omega_k^2}}
\end{equation}

\noindent Under the assumption that $\omega_k$ is sufficiently large such that $|4 P \sin\theta_k/\omega_k^2| < 1$, the square root in the above equation can be expanded using the convergent Taylor series to give
\begin{equation} \label{eq:omkp1-from-E-taylor}
    \omega_{k+1} = \omega_k + \frac{2 P \sin\theta_k}{\omega_k} - \frac{2 P^2 \sin^2\theta_k}{\omega_k^3} + \frac{4 P^3 \sin^3\theta_k}{\omega_k^5} - \frac{10 P^4 \sin^4\theta_k}{\omega_k^7} + \order{\frac{P^5 \sin^5\theta_k}{\omega_k^{9}}}
\end{equation}

\noindent The values of $\omega_k$ obtained from \eqref{eq:om-periodic-discrete} differ from the actual values of $\omega_k$ by an error $\epsilon_k$ for $k \geq 2$, as they are obtained assuming $\bar E$ to be invariant along system trajectories. Since $\bar E$ is obtained from initial conditions, $\epsilon_1 = 0$. Given $\epsilon_k$, the error $\epsilon_{k+1}$ is approximated by taking the difference between the series expansions \eqref{eq:omkp1-from-E-taylor} and \eqref{eq:omkp1-final-taylor}:
\begin{equation}
\begin{split}
        \epsilon_{k+1} = \epsilon_k &+ 2 P \sin\theta_k \left[ \frac{1}{\omega_k + \epsilon_k} - \frac{1}{\omega_k} \right] - 2 P^2 \sin^2\theta_k \left[ \frac{1}{(\omega_k + \epsilon_k)^3} - \frac{1}{\omega_k^3} \right] \\
        &+ P^3 \sin^3\theta_k \left[ \frac{4}{(\omega_k + \epsilon_k)^5} - \frac{6}{\omega_k^5} \right] - P^4 \sin^4\theta_k \left[ \frac{10}{(\omega_k + \epsilon_k)^7} - \frac{22}{\omega_k^7} \right] + \order{\frac{P^5 \sin^5\theta_k}{\omega_k^{9}}}
\end{split}
\end{equation}

\noindent The above series is convergent since it is obtained as a difference between two convergent series. Clearly, the larger the value of $\omega_k$, the smaller the incremental error.

\section{Simulation Results}

This section tabulates simulation results based on \eqref{eq:thkp1-zero-dyn} and \eqref{eq:omkp1-final} to illustrate the behavior of the dynamical system. We use the following parameters:
\begin{equation*}
    g = 9.81, \quad \ell = 1
\end{equation*}
\begin{table}[b!] \label{tab:integer}
\centering
\caption{Simulation results for $\delths = 2\pi/N$}
\scalebox{0.91}{
\begin{tabular}{|c|c|c|c|c|c|c|c|}
\hline
$N$                & $\theta_1$                   & $\omega_1$ & Steps                   & Notes                     & \begin{tabular}[c]{@{}c@{}}Maximum $\%$ error \\[-0.5ex] between exact $\omega_k$\\[-0.5ex] and value from \eqref{eq:om-periodic-discrete}\end{tabular} & \begin{tabular}[c]{@{}c@{}}$\%$ drift in $\omega_k$\\[-0.5ex] from $\omega_1$\end{tabular} & Illustration               \\ \hline\hline
\multirow{6}{*}{3} & \multirow{2}{*}{$0$}         & $12$       & \multirow{6}{*}{$1200$} & \multirow{4}{*}{Periodic} & $5.2507$                                                                                                                                & \multirow{4}{*}{N/A}                                                               & Fig.\ref{Fig:integer}(a) \\ \cline{3-3} \cline{6-6} \cline{8-8} 
                   &                              & $30$       &                         &                           & $3.3435 \times 10^{-3}$                                                                                                                 &                                                                                    &                      \\ \cline{2-3} \cline{6-6} \cline{8-8} 
                   & \multirow{2}{*}{$\delths/2$} & $12$       &                         &                           & $5.2507$                                                                                                                                &                                                                                    &                      \\ \cline{3-3} \cline{6-6} \cline{8-8} 
                   &                              & $30$       &                         &                           & $3.3435 \times 10^{-3}$                                                                                                                 &                                                                                    &                      \\ \cline{2-3} \cline{5-8} 
                   & \multirow{2}{*}{$0.1$}       & $12$       &                         & \multirow{2}{*}{}         & $-49.7107$                                                                                                                              & $44.9383$                                                                          & Fig.\ref{Fig:integer}(b) \\ \cline{3-3} \cline{6-8} 
                   &                              & $30$       &                         &                           & $-0.4389$                                                                                                                               & $0.4000$                                                                           & Fig.\ref{Fig:integer}(c) \\ \hline\hline
\multirow{6}{*}{4} & \multirow{2}{*}{$0$}         & $10$       & \multirow{4}{*}{$1200$} & \multirow{4}{*}{Periodic} & $1.4427$                                                                                                                                & \multirow{4}{*}{N/A}                                                               & Fig.\ref{Fig:integer}(d) \\ \cline{3-3} \cline{6-6} \cline{8-8} 
                   &                              & $30$       &                         &                           & $5.4332 \times 10^{-4}$                                                                                                                 &                                                                                    &                      \\ \cline{2-3} \cline{6-6} \cline{8-8} 
                   & \multirow{2}{*}{$\delths/2$} & $10$       &                         &                           & $2.8509$                                                                                                                                &                                                                                    &                      \\ \cline{3-3} \cline{6-6} \cline{8-8} 
                   &                              & $30$       &                         &                           & $4.0290 \times 10^{-4}$                                                                                                                 &                                                                                    &                      \\ \cline{2-8} 
                   & \multirow{2}{*}{$0.1$}       & $10$       & $212$                   & \multirow{2}{*}{}         & $65.6032$                                                                                                                               & $-13.5048$                                                                         & Fig.\ref{Fig:integer}(e) \\ \cline{3-4} \cline{6-8} 
                   &                              & $30$       & $1200$                  &                           & $2.3182 \times 10^{-3}$                                                                                                                 & $-1.6747 \times 10^{-3}$                                                           &                      \\ \hline\hline
\multirow{6}{*}{6} & \multirow{2}{*}{$0$}         & $ 8$       & \multirow{4}{*}{$1200$} & \multirow{4}{*}{Periodic} & $2.4646$                                                                                                                                & \multirow{4}{*}{N/A}                                                               & Fig.\ref{Fig:integer}(f), \ref{Fig:cobweb} \\ \cline{3-3} \cline{6-6} \cline{8-8} 
                   &                              & $30$       &                         &                           & $9.4188 \times 10^{-5}$                                                                                                                 &                                                                                    &                      \\ \cline{2-3} \cline{6-6} \cline{8-8} 
                   & \multirow{2}{*}{$\delths/2$} & $ 8$       &                         &                           & $4.2155$                                                                                                                                &                                                                                    &                      \\ \cline{3-3} \cline{6-6} \cline{8-8} 
                   &                              & $30$       &                         &                           & $9.2024 \times 10^{-5}$                                                                                                                 &                                                                                    &                      \\ \cline{2-8} 
                   & \multirow{2}{*}{$0.1$}       & $ 8$       & $516$                   & \multirow{2}{*}{}         & $68.2575$                                                                                                                               & $-7.4799$                                                                          &                      \\ \cline{3-4} \cline{6-8} 
                   &                              & $30$       & $1200$                  &                           & $9.4333 \times 10^{-5}$                                                                                                                 & $-5.0463 \times 10^{-8}$                                                           &                      \\ \hline
\end{tabular}
}
\end{table}
\noindent and assume the negative value of $P$ in \eqref{eq:P}. The simulations are carried out for multiple values of $\delths$. The initial conditions are chosen with $\theta_1 \in \{0, \delths/2, 0.1\}$ and two choices of $\omega_1$ for each $\theta_1$. We choose a small value of $\omega_1$ that permits solutions to \eqref{eq:omkp1-final} for all $k$; the other value is chosen to be equal to $30$, which is sufficiently large. It must be noted that the small value of $\omega_1$ may lead to \eqref{eq:bound-omk-real-sol-taylor} being violated while still permitting real, feasible solutions to \eqref{eq:omkp1-final}; in such cases, the series expansion \eqref{eq:omkp1-final-taylor} is not convergent. Simulations are carried out for a large number of steps, to study the long-term evolution of the system. To compare the behavior of the dynamical system with the approximate solution, $\bar E$ is evaluated from \eqref{eq:E-discrete} based on the initial conditions. The values of $\omega_k$ are then obtained from \eqref{eq:om-periodic-discrete}, and the maximum error with the exact solution is tabulated. Polar plots of $(\theta_k, \omega_k)$, with $\omega_k$ represented in the radial direction, are presented to show the evolution of system trajectories for a select few choices of initial conditions and $\delths$; the exact and approximate solutions are shown using black and red, respectively.\ 

\subsection{$\delths$: Integer submultiple of $2\pi$} \label{sec5.1}

We consider cases with $\delths$ given by \eqref{eq:delths}, for which it was established in Section \ref{sec:proof-periodicity} that the system exhibits periodic behavior for $\theta_1 \in \{0, \delths/2\}$ and $\omega_1$ satisfying Assumption \ref{asm:omk-min-periodic}. This is validated by simulation for $N = 3, 4, 6$. The solutions obtained satisfy Corollary \ref{cor:stable}, as expected. When $\theta_1 = 0.1$, it is observed that the system trajectories are not periodic; for $N = 3$, the value of $\omega_k$ increases gradually, and for $N = 4, 6$, $\omega_k$ decreases gradually. The simulation window in the latter two cases is reduced from 1200 since \eqref{eq:omkp1-final} ceases to permit solutions when $\omega_k$ reduces below a certain value. Similar behavior can be observed for other values of $\theta_1 \notin \{0, \delths/2\}$.

The simulation results are presented in Table \ref{tab:integer}. The third-to-last column shows the maximum percentage error between exact values of $\omega_k$ obtained from \eqref{eq:omkp1-final}, and values predicted by \eqref{eq:om-periodic-discrete}. As expected, these errors have larger magnitudes for smaller values of $\omega_1$. Further, the error magnitudes are greater when $\theta_1 = 0.1$ owing to the gradual `drift' in $\omega_k$ from $\omega_1$. The second-to-last column of Table \ref{tab:integer} shows the percentage drift in the value of $\omega_k$ from $\omega_1$ within the simulation window when $\theta_k$ returns to $\theta_1$. The drift is zero when $\theta_1 \in \{0, \delths/2\}$ since the system is periodic. When $\theta_1 = 0.1$, the magnitude of the drift is greater for lower values of $\omega_1$. It is also evident that the system behavior is closer to periodic for larger values of $N$ and larger values of $\omega_1$ if $\theta_1 \notin \{0, \delths/2\}$. Polar plots corresponding to a select few entries in Table \ref{tab:integer} are shown in Fig.\ref{Fig:integer}.
\begin{figure}[t!]
\centering
\psfrag{A}[][]{\scriptsize{$0$}}
\psfrag{B}[][]{\scriptsize{$\pi/2$}}
\psfrag{C}[][]{\scriptsize{$\pi$}}
\psfrag{D}[][]{\scriptsize{$3\pi/2$}}
\includegraphics[width=0.91\hsize]{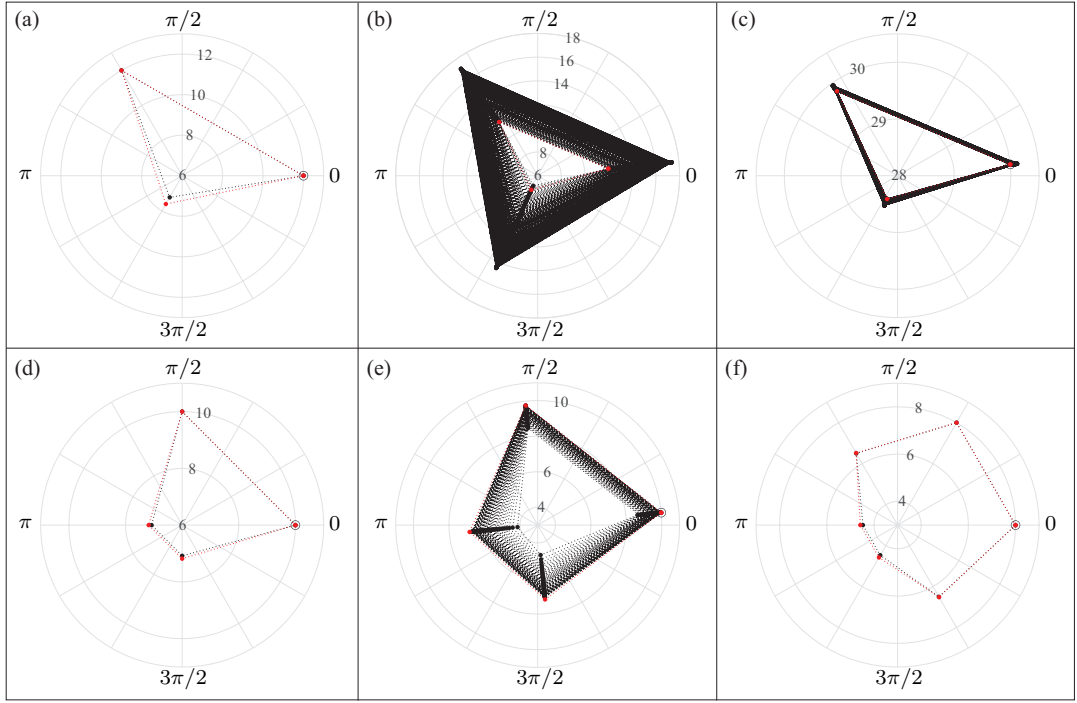}
\caption{Polar plots of $(\theta_k, \omega_k)$ for $\delths = 2\pi/N$ with $(N, \theta_1, \omega_1) =$ (a) $3, 0, 12$, (b) $3, 0.1, 12$, (c) $3, 0.1, 30$, (d) $4, 0, 10$, (e) $4, 0.1, 10$, (f) $6, 0, 8$.}
\label{Fig:integer}
\end{figure}
\begin{figure}[h!]
    \centering
    \includegraphics[width=0.99\linewidth]{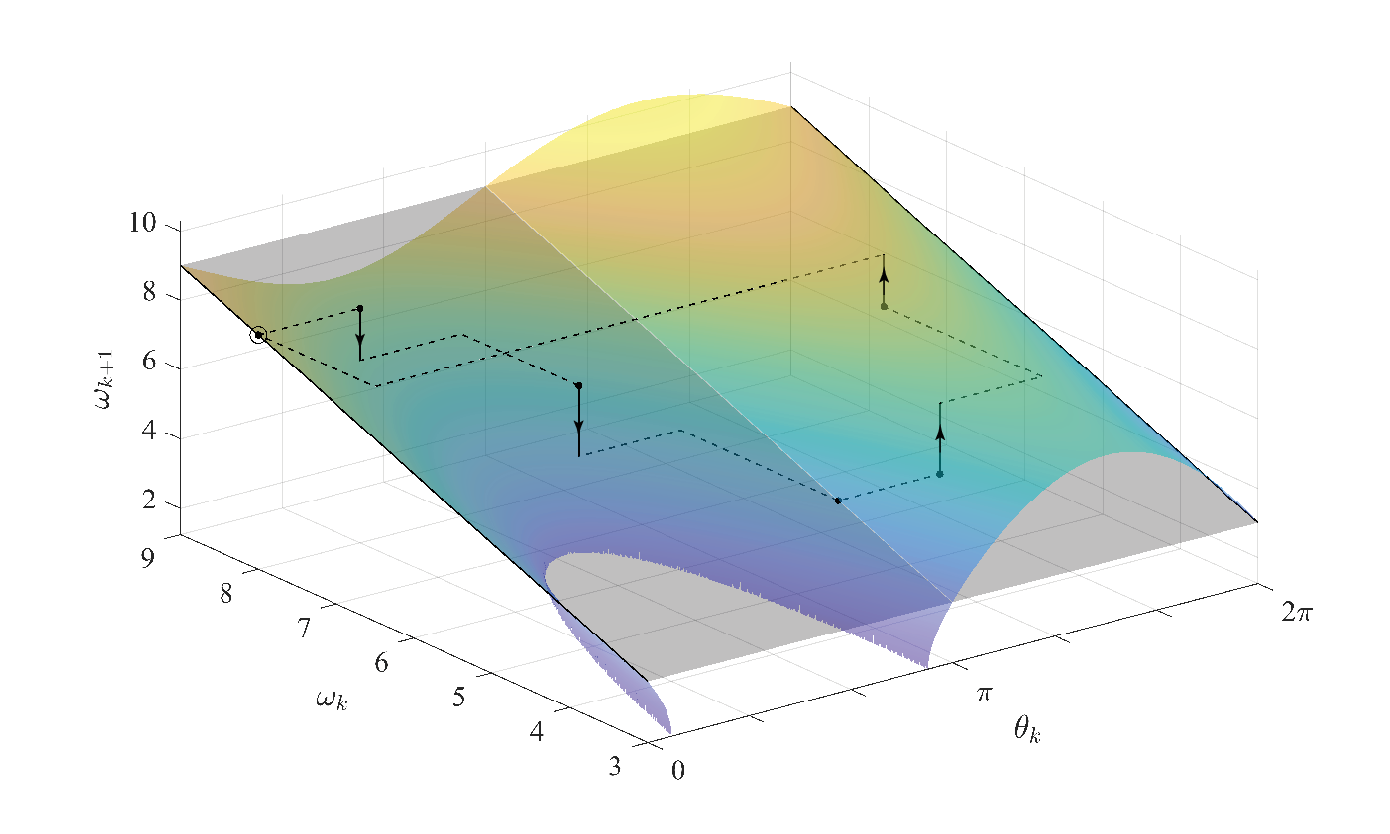}
    \caption{Cobweb-like plot corresponding to Fig.\ref{Fig:integer}(f): $\delths = \pi/3$ with $\theta_1 = 0$ and $\omega_1 = 8$.}
    \label{Fig:cobweb}
\end{figure}

The periodic orbit shown in Fig.\ref{Fig:integer}(f), \emph{i.e.}, case with $N = 6$, $\theta_1 = 0$ and $\omega_1 = 8$, is also shown with the help of the cobweb-like plot in Fig.\ref{Fig:cobweb}. With $\omega_{k+1}$ along the vertical axis, the surface defined by \eqref{eq:omkp1-final} is shown in color; it is not defined where \eqref{eq:bound-omk-real-sol} is violated. The surface $\omega_{k+1} = \omega_k$ is shown in gray. At $\theta_k \in \{0, \pi\}$, these two surfaces coincide. For any pair $(\theta_k, \omega_k)$ lying on the gray surface, the value of $\omega_{k+1}$ is found by drawing a vertical line onto the surface defined by \eqref{eq:omkp1-final} (solid line with arrows). The point on the gray surface for the next iteration is obtained by drawing a line of length $\delths$ along the $\theta_k$ axis followed by a line along the $\omega_k$ axis until it intersects the gray surface (dashed lines). The process can then be repeated to obtain a closed path describing the periodic orbit. As expected, the vertical arrows have length zero at $\theta_k = 0$ and $\theta_k = \pi$. Similar plots can be obtained for any other simulation case.

For the same periodic orbit discussed above, the Jacobian matrices $J_k$ in \eqref{eq:floquet-matrix} for $k = 1, 2, \dots, 6$, are given by
\begin{equation*}
\begin{split}
    J_1 &= \begin{bmatrix} 1 & 0 \\ -1.5528 & 1.0000 \end{bmatrix},\, 
    J_2 = \begin{bmatrix} 1 & 0 \\ -0.9923 & 1.2422 \end{bmatrix},\, 
    J_3 = \begin{bmatrix} 1 & 0 \\  1.6049 & 1.5428 \end{bmatrix}, \\
    J_4 &= \begin{bmatrix} 1 & 0 \\  2.7796 & 1.0000 \end{bmatrix},\, 
    \,\,\,\,\,J_5 = \begin{bmatrix} 1 & 0 \\  1.0402 & 0.6482 \end{bmatrix},\, 
    \,\,\,\,J_6 = \begin{bmatrix} 1 & 0 \\ -0.7988 & 0.8050 \end{bmatrix}
\end{split}
\end{equation*}

\noindent which result in 
\begin{equation*}
    M = \begin{bmatrix} 1 & 0 \\ -0.0252 & 1.0000 \end{bmatrix}
\end{equation*}

\noindent Both eigenvalues of $M$ are equal to $1$, implying stability of the orbit.

\subsection{$\delths$: Rational submultiple of $2\pi$} \label{sec5.2}

We consider cases with $\delths = (p/q)2\pi$, where $p, q \in Z^+$, and $\gcd(p, q) = 1$. Simulation results for $p/q = 3/7$, $2/9$, and $4/21$ are presented in Table \ref{tab:rational}. Here, it is not known whether the system is periodic for any initial conditions. However, for $\theta_1 \in \{0, \delths/2\}$, the behavior appears to be periodic with negligible drift in the magnitude of $\omega_k$ from $\omega_1$ when $\theta_k$ returns to $\theta_1$, for all choices of $p/q$ and $\omega_1$. For $\theta_1 = 0.1$, a significant drift in $\omega_k$ is observed for $p/q = 3/7$ and $p/q = 2/9$ for the smaller value of $\omega_1$; for $p/q = 4/21$, the drift is relatively small for the smaller value of $\omega_1$. For each value of $p/q$, the magnitude of the drift is much smaller for $\omega_1 = 30$. Further, as expected, the error between the exact value of of $\omega_k$ and its value predicted by \eqref{eq:om-periodic-discrete} is smaller for larger values of $\omega_1$ in all cases. The behavior of the system is closer to periodic as $p/q$ decreases and $\omega_1$ increases. Polar plots corresponding to to a select few entries in Table \ref{tab:rational} are shown in Fig.\ref{Fig:rational}.

\begin{table}[t!] \label{tab:rational}
\centering
\caption{Simulation Results for $\delths = (p/q)2\pi$}
\scalebox{0.91}{
\begin{tabular}{|c|c|c|c|c|c|c|c|}
\hline
$p/q$                   & $\theta_1$                   & $\omega_1$ & Steps                   & Notes                     & \begin{tabular}[c]{@{}c@{}}Maximum $\%$ error \\[-0.5ex] between exact $\omega_k$\\[-0.5ex] and value from \eqref{eq:om-periodic-discrete}\end{tabular} & \begin{tabular}[c]{@{}c@{}}$\%$ drift in $\omega_k$\\[-0.5ex] from $\omega_1$\end{tabular} & Illustration              \\ \hline\hline
\multirow{6}{*}{$3/7$}  & \multirow{2}{*}{$0$}         & $19$       & \multirow{6}{*}{$1200$} & \multirow{4}{*}{Periodic} & $10.9992$                                                                                                                               & $-7.1019 \times 10^{-9}$                                                           & Fig.\ref{Fig:rational}(a) \\ \cline{3-3} \cline{6-8} 
                        &                              & $30$       &                         &                           & $0.1312$                                                                                                                                & $-7.7179 \times 10^{-10}$                                                          &                           \\ \cline{2-3} \cline{6-8} 
                        & \multirow{2}{*}{$\delths/2$} & $22$       &                         &                           & $13.5781$                                                                                                                               & $8.2037 \times 10^{-9}$                                                            &                           \\ \cline{3-3} \cline{6-8} 
                        &                              & $30$       &                         &                           & $0.3451$                                                                                                                                & $5.9093 \times 10^{-12}$                                                           &                           \\ \cline{2-3} \cline{5-8} 
                        & \multirow{2}{*}{$0.1$}       & $19$       &                         & \multirow{2}{*}{}         & $18.1939$                                                                                                                               & $8.4108$                                                                           & Fig.\ref{Fig:rational}(b) \\ \cline{3-3} \cline{6-8} 
                        &                              & $30$       &                         &                           & $0.1383$                                                                                                                                & $4.5355 \times 10^{-3}$                                                            & Fig.\ref{Fig:rational}(c) \\ \hline\hline
\multirow{6}{*}{$2/9$}  & \multirow{2}{*}{$0$}         & $ 8$       & \multirow{6}{*}{$1200$} & \multirow{4}{*}{Periodic} & $16.1407$                                                                                                                               & $7.1658 \times 10^{-9}$                                                            & Fig.\ref{Fig:rational}(d) \\ \cline{3-3} \cline{6-8} 
                        &                              & $30$       &                         &                           & $2.7761 \times 10^{-4}$                                                                                                                 & $-2.1790 \times 10^{-11}$                                                          &                           \\ \cline{2-3} \cline{6-8} 
                        & \multirow{2}{*}{$\delths/2$} & $ 9$       &                         &                           & $6.1096$                                                                                                                                & $2.9724 \times 10^{-11}$                                                           &                           \\ \cline{3-3} \cline{6-8} 
                        &                              & $30$       &                         &                           & $2.6540 \times 10^{-4}$                                                                                                                 & $7.1054 \times 10^{-14}$                                                           &                           \\ \cline{2-3} \cline{5-8} 
                        & \multirow{2}{*}{$0.1$}       & $ 9$       &                         & \multirow{2}{*}{}         & $2.8702$                                                                                                                                & $-0.1268$                                                                          & Fig.\ref{Fig:rational}(e) \\ \cline{3-3} \cline{6-8} 
                        &                              & $30$       &                         &                           & $2.8351 \times 10^{-4}$                                                                                                                 & $-1.9457 \times 10^{-11}$                                                          &                           \\ \hline\hline
\multirow{6}{*}{$4/21$} & \multirow{2}{*}{$0$}         & $ 8$       & \multirow{6}{*}{$1200$} & \multirow{4}{*}{Periodic} & $7.4877$                                                                                                                                & $3.3809 \times 10^{-10}$                                                           & Fig.\ref{Fig:rational}(f) \\ \cline{3-3} \cline{6-8} 
                        &                              & $30$       &                         &                           & $1.5035 \times 10^{-4}$                                                                                                                 & $2.3791 \times 10^{-11}$                                                           &                           \\ \cline{2-3} \cline{6-8} 
                        & \multirow{2}{*}{$\delths/2$} & $ 8$       &                         &                           & $11.7000$                                                                                                                               & $-1.2212 \times 10^{-12}$                                                          &                           \\ \cline{3-3} \cline{6-8} 
                        &                              & $30$       &                         &                           & $1.4776 \times 10^{-4}$                                                                                                                 & $-3.7896 \times 10^{-13}$                                                          &                           \\ \cline{2-3} \cline{5-8} 
                        & \multirow{2}{*}{$0.1$}       & $ 8$       &                         & \multirow{2}{*}{}         & $9.4883$                                                                                                                                & $-6.8383 \times 10^{-4}$                                                           &                           \\ \cline{3-3} \cline{6-8} 
                        &                              & $30$       &                         &                           & $1.5133 \times 10^{-4}$                                                                                                                 & $2.0416 \times 10^{-11}$                                                           &                           \\ \hline
\end{tabular}}
\end{table}

\begin{figure}[t!]
\centering
\psfrag{A}[][]{\scriptsize{$0$}}
\psfrag{B}[][]{\scriptsize{$\pi/2$}}
\psfrag{C}[][]{\scriptsize{$\pi$}}
\psfrag{D}[][]{\scriptsize{$3\pi/2$}}
\includegraphics[width=0.91\hsize]{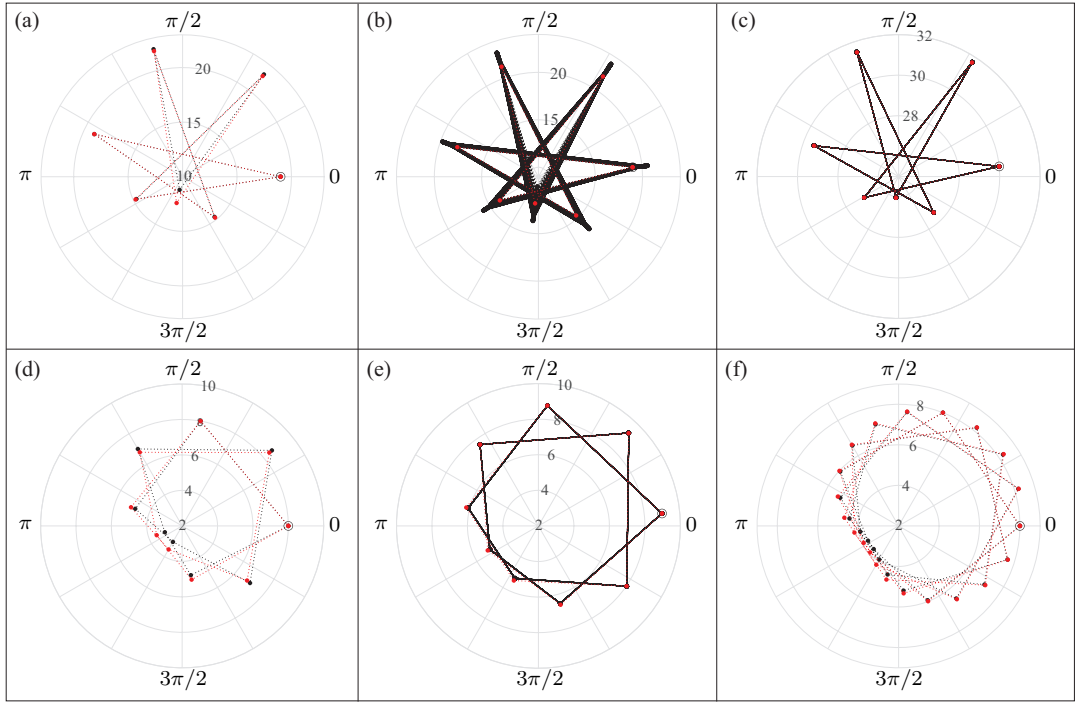}
\caption{Polar plots of $(\theta_k, \omega_k)$ for $\delths = (p/q) 2\pi$ with $(N, \theta_1, \omega_1) =$ (a) $3/7, 0, 19$, (b) $3/7, 0.1, 19$, (c) $3/7, 0.1, 30$, (d) $2/9, 0, 8$, (e) $2/9, 0.1, 9$, (f) $4/21, 0, 8$.}
\label{Fig:rational}
\end{figure}

\subsection{$\delths$: Irrational submultiple of $2\pi$} \label{sec5.3}

We now consider the case $\delths = (\sqrt{2}/5) 2\pi$. The map \eqref{eq:thkp1-zero-dyn} then represents an irrational rotation, and contains no periodic points. Consequently, the overall system \eqref{eq:thkp1-zero-dyn} and \eqref{eq:omkp1-final} can have no periodic points. Over time, however, the iterations $\theta_k$ densely occupy points in $S^1$. We simulate the system over 1200 steps with $\theta_1 = 0$ and $\omega_1 = 10$, $20$, and $30$. Since $\theta_k$ values will densely fill the circle, the qualitative behavior of the system will be the same regardless of $\theta_1$, and different choices of $\theta_1$ need not be considered.

When $\omega_1 = 10$, the maximum error between the exact $\omega_k$ and its values predicted from \eqref{eq:om-periodic-discrete} is $8.1788 \%$. This error drops to $1.7975 \times 10^{-2} \%$ for $\omega_1 = 20$, and $1.3405 \times 10^{-3} \%$ for $\omega_1 = 30$. When $\omega_1 = 10$, the maximum $\omega_k$ in simulation is found to be $10.5850$ at $\theta_k = 0.8883 \approx \delths/2$, and the minimum $\omega_k$ is $5.3506$ at $\theta_k = 4.0319 \approx \pi + \delths/2$, as predicted by \eqref{eq:om-periodic-discrete}. The maxima and minima of $\omega_k$ for the other choices of $\omega_1$ also occur at $\theta_k \approx \delths/2$ and $\theta_k \approx \pi + \delths/2$ respectively. The polar plot for the first 60 simulation steps with $\omega_1 = 10$ is shown in Fig.\ref{Fig3} (a).

\begin{remark}
    The results in Sections \ref{sec5.1}, \ref{sec5.2}, and \ref{sec5.3} demonstrate that \eqref{eq:om-periodic-discrete} accurately captures the qualitative behavior of the dynamical system, and, for sufficiently large values of $\omega_k$, the quantitative behavior of the system.
\end{remark}

\subsection{Limiting value of $\delths$} \label{sec5.4}

We consider the limiting case $\delths \rightarrow 0$, which is achieved in simulation by considering a large value of $N$ in \eqref{eq:delths}. The behavior of the system is expected to match the behavior of \eqref{eq:simple-pendulum}. 

For the initial conditions $\theta_1 = 0$ and $\omega_1 = 6.4$, we simulate \eqref{eq:thkp1-zero-dyn} and \eqref{eq:omkp1-final} for 1200 steps with $N = 120$, or $\delths = 0.0524$. The maximum error between the exact $\omega_k$ and its values predicted from \eqref{eq:om-periodic-discrete} is $0.2882 \%$, which is small considering the small initial $\omega_1$. Further, the maximum error between the exact $\omega_k$ and its values predicted from \eqref{eq:om-periodic-pendulum}, with $E$ computed using \eqref{eq:integral-of-motion-pendulum}, is $3.5168 \%$. These errors drop to $2.8311 \times 10^{-3} \%$ and $0.3100 \%$ respectively when $N$ is increased to $1200$, \emph{i.e.} $\delths = 0.0052$, for the same initial conditions. The polar plot for the case $N = 120$ is shown in Fig.\ref{Fig3} (b).

Further, if $\theta_1$ is changed to $0.1$ (or any value different from $0$ or $\delths/2$), the observed magnitude of the drift in $\omega_k$ from $\omega_1$ (when $\theta_k$ returns to $\theta_1$) is negligibly small even when the number of simulation steps is increased drastically, implying periodic behavior.\

\begin{remark}
    The results in Sections \ref{sec5.1}, \ref{sec5.2}, \ref{sec5.3} and \ref{sec5.4} all assume the negative value of $P$. For every initial condition considered, the positive square root solution in \eqref{eq:omkp1-final} provides feasible solutions. In particular, when $\delths$ is given by \eqref{eq:delths}, \eqref{eq:omkp1-final} provides the periodic solution when $\theta_1 \in \{0, \delths/2\}$. 
\end{remark}

\begin{figure}[t!]
\centering
\psfrag{A}[][]{\scriptsize{$0$}}
\psfrag{B}[][]{\scriptsize{$\pi/2$}}
\psfrag{C}[][]{\scriptsize{$\pi$}}
\psfrag{D}[][]{\scriptsize{$3\pi/2$}}
\includegraphics[width=0.91\hsize]{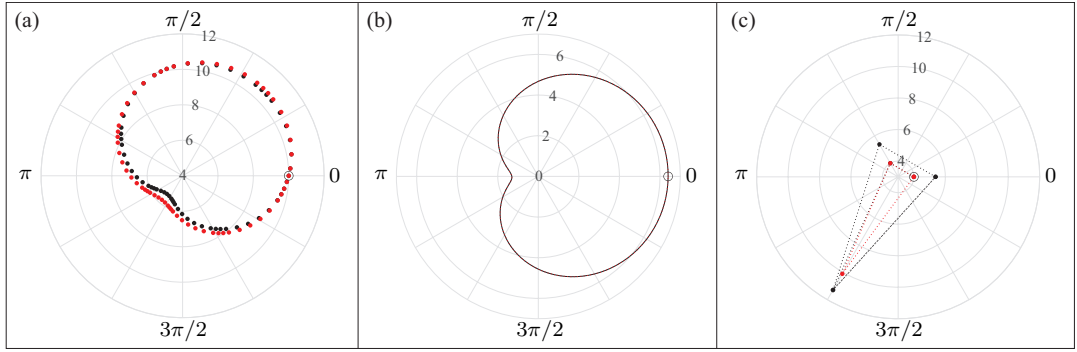}
\caption{Polar plots of $(\theta_k, \omega_k)$ for (a) $\delths = (\sqrt{2}/5) 2\pi$, $\theta_1 = 0$, and $\omega_1 = 10$, (b) $\delths = \pi/60$, $\theta_1 = 0$, and $\omega_1 = 6.4$, (c) $\delths = 2\pi/3$, $\theta_1 = 0$, and $\omega_1 = 4$, when $P$ is positive.}
\label{Fig3}
\end{figure}
\begin{table}[b!] \label{tab:posP}
\centering
\caption{Simulation Results for $\delths = 2\pi/3$ when $P$ is +ve}
\scalebox{0.96}{
\begin{tabular}{|c|c|c|c|c|c|c|}
\hline
$N$ &
  $\theta_1$ &
  $\omega_1$ &
  Steps &
  Notes &
  \begin{tabular}[c]{@{}c@{}}Maximum $\%$ error \\[-0.5ex] between exact $\omega_k$\\[-0.5ex] and value from \eqref{eq:om-periodic-discrete}\end{tabular} &
  \begin{tabular}[c]{@{}c@{}}$\%$ drift in $\omega_k$\\ [-0.5ex] from $\omega_1$\end{tabular} \\ \hline\hline
\multirow{6}{*}{3} & \multirow{2}{*}{$0$}         & $12$ & \multirow{4}{*}{$1200$} & \multirow{4}{*}{Periodic} & $-0.2626$                & \multirow{4}{*}{N/A} \\ \cline{3-3} \cline{6-6}
                   &                              & $30$ &                         &                           & $-2.2774 \times 10^{-3}$ &                      \\ \cline{2-3} \cline{6-6}
                   & \multirow{2}{*}{$\delths/2$} & $12$ &                         &                           & $-0.2626$                &                      \\ \cline{3-3} \cline{6-6}
                   &                              & $30$ &                         &                           & $-2.2774 \times 10^{-3}$ &                      \\ \cline{2-7} 
                   & \multirow{2}{*}{$0.1$}       & $12$ & $300$                   & \multirow{2}{*}{}         & $57.5809$                & $-36.5406$           \\ \cline{3-4} \cline{6-7} 
                   &                              & $30$ & $1200$                  &                           & $0.3522$                 & $-0.3509$            \\ \hline
\end{tabular}}
\end{table}

\subsection{Positive Choice of $P$}

We now consider the positive value of $P$ in \eqref{eq:P} and simulate cases with $\delths = 2\pi/3$ and $\delths = (\sqrt{2}/5) 2\pi$.

Simulation results for $\delths = 2\pi/3$ are shown in Table \ref{tab:posP}, for the same initial conditions as were considered in Table \ref{tab:integer} for $N = 3$. Again, periodic behavior is seen for the positive square root solution in \eqref{eq:omkp1-final} when $\theta_1 \in \{0, \delths/2\}$, and a drift in $\omega_k$ is seen when $\theta_1 = 0.1$. The magnitude of the drift is also smaller when $\omega_1$ is larger.

For $\delths = 2\pi/3$, we also consider the initial condition $\theta_1 = 0$ and $\omega_1 = 4$, for which Assumption \ref{asm:omk-min-periodic} is violated, since $\omega_1 < \sqrt{P}$. In this case, using the negative square root solution in \eqref{eq:omkp1-quadratic-sol} at $\theta_k = 4\pi/3$ and the positive square root solution in \eqref{eq:omkp1-final} for all other values of $\theta_k$ provides the periodic solution. Using the positive square root solution at $k = 3$ results in $\omega_4 = 5.3789 \neq \omega_1$. However, solutions from these initial conditions, \emph{i.e.}, $\theta_4 = \theta_1$ and $\omega_4$ are periodic using the positive square root solution in \eqref{eq:omkp1-final}; this is true because $\omega_4$ no longer violates Assumption \ref{asm:omk-min-periodic}. The polar plot for this case in shown in Fig.\ref{Fig3} (c).

We finally consider $\delths = (\sqrt{2}/5) 2\pi$ and simulate 1200 steps. With the initial conditions $\theta_1 = 0$ and $\omega_1 = 10$, the maximum error between the exact $\omega_k$ and its values from \eqref{eq:om-periodic-discrete} is $-0.3894 \%$. The minimum and maximum values of $\omega_k$, $9.2234$ and $12.9270$, are achieved at $\theta_k = 0.8883 \approx \delths/2$ and $\theta_k = 4.0319 \approx \pi + \delths/2$ respectively.

The change in the sign of $P$ affects the qualitative behavior of the system, best explained by \eqref{eq:om-periodic-discrete}. In particular, the change in the sign of $P$ causes the minima and the maxima to change locations.

\section{Conclusion} \label{sec6}
This paper considered a two-dimensional discrete-time dynamical system with the two states representing an angle and angular velocity. 
When the angle evolves by an integer submultiple of $2\pi$ in every iteration, it is proved that the system is periodic for specific initial angles but arbitrary initial angular velocities greater than some minimum value. Every such periodic trajectory is stable but not attractive. For other initial angles, simulations show that periodicity is lost. 

As the angle step size tends to zero, the discrete-time system represents the dynamics of a simple pendulum, which possesses an integral of motion. Based on this integral of motion, an approximate invariant for the discrete-time dynamical system is found; this allows the angular velocity to be expressed as an explicit periodic function of the angle. It is shown that this expression closely approximates the behavior of the dynamical system for sufficiently large angular velocities. 

Extensive simulation results document the behavior of the dynamical system for a range of angle step sizes and initial conditions, showing periodic and aperiodic behavior of the system. 

\bibliographystyle{siamplain}
\bibliography{references}

\newpage

\appendix
\section{Proof of Lemma \ref{lem:delths}} \label{app:proof}

It follows from \eqref{eq:thkp1-zero-dyn} and \eqref{eq:delths} that $\theta_{N+1} = \theta_1 + 2\pi = \theta_1$, which establishes periodicity of $\theta$. For any value of $N$ in \eqref{eq:delths}, we can express $\theta_k$ as follows
\begin{equation} \label{eq:thk-from-pi-delths}
    \theta_k = \frac{\delths}{2} + (k-1) \delths = \pi + (2k - N - 1) \frac{\delths}{2}
\end{equation}

\noindent from which it follows that
\begin{equation} \label{eq:sin-from-pi-delths}
    \sin\theta_k = - \sin\left[(2k - N - 1) \frac{\delths}{2}\right]
\end{equation}

\noindent Further, using \eqref{eq:sin-from-pi-delths} in \eqref{eq:omkp1-zero-dyn-reduced}, we get
\begin{equation} \label{eq:omkp1-zero-dyn-from-pi-delths}
    \omega_{k+1} - \omega_k = - P \sin\left[(2k - N - 1) \frac{\delths}{2}\right] \left[ \frac{1}{\omega_k} + \frac{1}{\omega_{k+1}} \right]
\end{equation}

\noindent We prove periodicity of $\omega$ separately for even and odd choices of $N$.

\paragraph{Even $N$}

Substituting $k = \frac{N}{2}$ and $k = \frac{N}{2}+1$ in \eqref{eq:omkp1-zero-dyn-from-pi-delths}, we obtain
\begin{align} 
    \omega_{\frac{N}{2}+1} - \omega_{\frac{N}{2}} &= P \sin\left(\frac{\delths}{2}\right) \left[ \frac{1}{\omega_{\frac{N}{2}}} + \frac{1}{\omega_{\frac{N}{2}+1}} \right] \label{eq:even-Nby2-delths} \\
    \omega_{\frac{N}{2}+2} - \omega_{\frac{N}{2}+1} &= - P \sin\left(\frac{\delths}{2}\right) \left[ \frac{1}{\omega_{\frac{N}{2}+1}} + \frac{1}{\omega_{\frac{N}{2}+2}} \right] \label{eq:even-Nby2p1-delths}
\end{align}

\noindent Adding the above two equations, we obtain
\begin{equation}
    \omega_{\frac{N}{2}+2} - \omega_{\frac{N}{2}} = P \sin\left(\frac{\delths}{2}\right) \left[ \frac{1}{\omega_{\frac{N}{2}}} - \frac{1}{\omega_{\frac{N}{2}+2}} \right]
\end{equation}
which simplifies to
\begin{equation}
    \omega_{\frac{N}{2}+2} - \omega_{\frac{N}{2}} = P \sin\left(\frac{\delths}{2}\right) \left[ \frac{\omega_{\frac{N}{2}+2} - \omega_{\frac{N}{2}}}{\omega_{\frac{N}{2}}\omega_{\frac{N}{2}+2}} \right]
\end{equation}
From Assumption \ref{asm:omk-min-periodic}, it follows
\begin{equation} \label{eq:equal-delthsby2-delths}
    \omega_{\frac{N}{2}+2} = \omega_{\frac{N}{2}}
\end{equation}

\noindent Similarly, choosing $k = \frac{N}{2}-1$ and $k = \left(\frac{N}{2}+1\right) + 1$ in \eqref{eq:omkp1-zero-dyn-from-pi-delths}, we obtain
\begin{align} 
    \omega_{\frac{N}{2}} - \omega_{\frac{N}{2}-1} &= P \sin\left(\frac{3\delths}{2}\right) \left[ \frac{1}{\omega_{\frac{N}{2}-1}} + \frac{1}{\omega_{\frac{N}{2}}} \right] \label{eq:even-Nby2m1-delths} \\
    \omega_{\frac{N}{2}+3} - \omega_{\frac{N}{2}+2} &= - P \sin\left(\frac{3\delths}{2}\right) \left[ \frac{1}{\omega_{\frac{N}{2}+2}} + \frac{1}{\omega_{\frac{N}{2}+3}} \right] \label{eq:even-Nby2p2-delths}
\end{align}

\noindent Adding the above two equations, we obtain
\begin{equation}
    \omega_{\frac{N}{2}+3} - \omega_{\frac{N}{2}+2} + \omega_{\frac{N}{2}} - \omega_{\frac{N}{2}-1} = P \sin\left(\frac{3\delths}{2}\right) \left[ \frac{1}{\omega_{\frac{N}{2}-1}} + \frac{1}{\omega_{\frac{N}{2}}} - \frac{1}{\omega_{\frac{N}{2}+2}} - \frac{1}{\omega_{\frac{N}{2}+3}} \right]
\end{equation}

\noindent Substituting \eqref{eq:equal-delthsby2-delths} in the above equation and simplifying, we obtain
\begin{equation}
    \omega_{\frac{N}{2}+3} - \omega_{\frac{N}{2}-1} = P \sin\left(\frac{3\delths}{2}\right) \left[ \frac{\omega_{\frac{N}{2}+3} - \omega_{\frac{N}{2}-1}}{\omega_{\frac{N}{2}-1}\omega_{\frac{N}{2}+3}} \right]
\end{equation}
Using Assumption \ref{asm:omk-min-periodic}, we get
\begin{equation} \label{eq:equal-3delthsby2-delths}
    \omega_{\frac{N}{2}+3} = \omega_{\frac{N}{2}-1}
\end{equation}

This process can be repeated for all pairs of values of $k = \frac{N}{2} - n$, $k = \left(\frac{N}{2} + 1\right) + n$, $n = 0, 1, \dots, \left(\frac{N}{2}-1\right)$ to establish that 
\begin{equation}  \label{eq:om-pairs-even-delths}
    \omega_{\frac{N}{2}+n+2} = \omega_{\frac{N}{2}-n}, \quad n = 0, 1, \dots, \left(\frac{N}{2}-1\right)
\end{equation}

\noindent In particular, if $n = \left(\frac{N}{2}-1\right)$ in the above equation, it follows that
\begin{equation} \label{eq:om-periodicity-even-delths}
    \omega_{N+1} = \omega_1
\end{equation}

\paragraph{Odd $N$}

Substituting $k = \frac{N+1}{2}$ in \eqref{eq:thk-from-pi-delths} and \eqref{eq:omkp1-zero-dyn-from-pi-delths}, we get $\theta_{\frac{N+1}{2}} = \pi$ and $\omega_{\frac{N+3}{2}} - \omega_{\frac{N+1}{2}} = 0$, from which it follows that
\begin{equation} \label{eq:equal-pi-delths}
    \omega_{\frac{N+3}{2}} = \omega_{\frac{N+1}{2}}
\end{equation}

\noindent Choosing $k = \frac{N+1}{2} - 1$ and $k = \frac{N+1}{2} + 1$ in \eqref{eq:omkp1-zero-dyn-from-pi-delths}, we obtain
\begin{align} 
    \omega_{\frac{N+1}{2}} - \omega_{\frac{N-1}{2}} &= P \sin(\delths) \left[ \frac{1}{\omega_{\frac{N-1}{2}}} + \frac{1}{\omega_{\frac{N+1}{2}}} \right] \label{eq:odd-Nm1by2-delths} \\
    \omega_{\frac{N+5}{2}} - \omega_{\frac{N+3}{2}} &= - P \sin(\delths) \left[ \frac{1}{\omega_{\frac{N+3}{2}}} + \frac{1}{\omega_{\frac{N+5}{2}}} \right] \label{eq:odd-Np3by2-delths}
\end{align}

\noindent Adding the above two equations, we obtain
\begin{equation}
    \omega_{\frac{N+5}{2}} - \omega_{\frac{N+3}{2}} + \omega_{\frac{N+1}{2}} - \omega_{\frac{N-1}{2}} = P \sin(\delths) \left[ \frac{1}{\omega_{\frac{N-1}{2}}} + \frac{1}{\omega_{\frac{N+1}{2}}} - \frac{1}{\omega_{\frac{N+3}{2}}} - \frac{1}{\omega_{\frac{N+5}{2}}} \right]
\end{equation}

\noindent Using \eqref{eq:equal-pi-delths}, the above equation simplifies to
\begin{equation}
    \omega_{\frac{N+5}{2}} - \omega_{\frac{N-1}{2}} = P \sin(\delths) \left[ \frac{\omega_{\frac{N+5}{2}} - \omega_{\frac{N-1}{2}}}{\omega_{\frac{N-1}{2}} \omega_{\frac{N+5}{2}}} \right]
\end{equation}
From Assumption \ref{asm:omk-min-periodic}, it follows
\begin{equation} \label{eq:equal-delths-delths}
    \omega_{\frac{N+5}{2}} = \omega_{\frac{N-1}{2}}
\end{equation}

\noindent Similarly, for $k = \frac{N+1}{2} - 2$ and $k = \frac{N+1}{2} + 2$ in \eqref{eq:omkp1-zero-dyn-from-pi-delths}, we obtain
\begin{align} 
    \omega_{\frac{N-1}{2}} - \omega_{\frac{N-3}{2}} &= P \sin(2\delths) \left[ \frac{1}{\omega_{\frac{N-3}{2}}} + \frac{1}{\omega_{\frac{N-1}{2}}} \right] \label{eq:odd-Nm3by2-delths} \\
    \omega_{\frac{N+7}{2}} - \omega_{\frac{N+5}{2}} &= - P \sin(2\delths) \left[ \frac{1}{\omega_{\frac{N+5}{2}}} + \frac{1}{\omega_{\frac{N+7}{2}}} \right] \label{eq:odd-Np5by2-delths}
\end{align}

\noindent Adding the above two equations, we obtain
\begin{equation}
    \omega_{\frac{N+7}{2}} - \omega_{\frac{N+5}{2}} + \omega_{\frac{N-1}{2}} - \omega_{\frac{N-3}{2}} = P \sin(2\delths) \left[ \frac{1}{\omega_{\frac{N-3}{2}}} + \frac{1}{\omega_{\frac{N-1}{2}}} - \frac{1}{\omega_{\frac{N+5}{2}}} - \frac{1}{\omega_{\frac{N+7}{2}}} \right]
\end{equation}

\noindent Using \eqref{eq:equal-delths-delths}, the above equation simplifies to
\begin{equation}
    \omega_{\frac{N+7}{2}} - \omega_{\frac{N-3}{2}} = P \sin(2\delths) \left[ \frac{\omega_{\frac{N+7}{2}} - \omega_{\frac{N-3}{2}}}{\omega_{\frac{N-3}{2}} \omega_{\frac{N+7}{2}}} \right]
\end{equation}
From Assumption \ref{asm:omk-min-periodic}, it follows
\begin{equation} \label{eq:equal-2delths-delths}
    \omega_{\frac{N+7}{2}} = \omega_{\frac{N-3}{2}}
\end{equation}

This process can be repeated for all pairs of values of $k = \frac{N+1}{2} - n$, $k = \frac{N+1}{2} + n$, $n = 0, 1, \dots, \left(\frac{N+1}{2}-1\right)$ to establish that 
\begin{equation}  \label{eq:om-pairs-odd-delths}
    \omega_{\frac{N+1}{2}+n+1} = \omega_{\frac{N+1}{2}-n}, \quad n = 0, 1, \dots, \left(\frac{N+1}{2}-1\right)
\end{equation}

\noindent In particular, if $n = \left(\frac{N+1}{2}-1\right)$ in the above equation, it follows that
\begin{equation} \label{eq:om-periodicity-odd-delths}
    \omega_{N+1} = \omega_1
\end{equation}

\noindent This concludes the proof.
\end{document}